\documentclass[12pt]{amsart}
\usepackage{fullpage}
\usepackage{units}
\usepackage{amsbsy}
\usepackage{amstext}
\usepackage{amsthm}
\usepackage{amssymb}
\usepackage{setspace}
\usepackage{bbm}
\usepackage[colorlinks,allcolors=blue]{hyperref}
\makeatletter
\theoremstyle{plain}
\newtheorem{thm}{\protect\theoremname}[section]
  \theoremstyle{definition}
  \newtheorem{defn}[thm]{\protect\definitionname}
  \theoremstyle{remark}
  \newtheorem*{rem*}{\protect\remarkname}
  \theoremstyle{remark}
  \newtheorem{rem}[thm]{\protect\remarkname}
  \theoremstyle{plain}
  \newtheorem{cor}[thm]{\protect\corollaryname}
  \theoremstyle{plain}
  \newtheorem{lem}[thm]{\protect\lemmaname}

\makeatother

  \providecommand{\corollaryname}{Corollary}
  \providecommand{\definitionname}{Definition}
  \providecommand{\lemmaname}{Lemma}
  \providecommand{\remarkname}{Remark}
\providecommand{\theoremname}{Theorem}

\DeclareMathOperator{\spa}{span}

\begin{document}
\global\long\def\SS{\mathcal{S}}
\global\long\def\TT{\mathbb{T}}
\global\long\def\norm#1{\left\Vert #1\right\Vert }
\global\long\def\ZZ{\mathbb{Z}}
\global\long\def\RR{\mathbb{R}}
\global\long\def\eps{\varepsilon}
\global\long\def\CC{\mathbb{C}}
\global\long\def\NN{\mathbb{N}}
\global\long\def\HH{\mathcal{H}}

\title{On syndetic Riesz sequences}

\author{Marcin Bownik}

\author{Itay Londner}

\address{Department of Mathematics, University of Oregon, Eugene, OR 97403--1222, USA}
\email{mbownik@uoregon.edu}

\address{School of Mathematical Sciences, Tel-Aviv University, Tel-Aviv 69978, Israel}
\curraddr{Department of Mathematics, University of British Columbia,
Vancouver, BC, V6T 1Z2,
Canada}

\email{itayl@math.ubc.ca}

\thanks{The first author was supported in part by the NSF grant DMS-1665056. The authors are grateful to Prof. Olevskii for discussions on the subject of this paper.}

\begin{abstract}
Applying the solution to the Kadison-Singer problem, we show
that every subset $\mathcal{S}$ of the torus of positive Lebesgue
measure admits a Riesz sequence of exponentials $\left\{ e^{i\lambda x}\right\} _{\lambda \in \Lambda}$
such that $\Lambda\subset\mathbb{Z}$ is a set with gaps between consecutive
elements bounded by ${\displaystyle \frac{C}{\left|\mathcal{S}\right|}}$.
In the case when $\mathcal{S}$ is an open set we demonstrate, using quasicrystals, 
how such $\Lambda$ can be deterministically constructed.	
\end{abstract}

\maketitle

\section{Introduction}

Let $\mathcal{H}$ be a separable Hilbert space and $I$
a countable set.

\begin{defn}
	A system of vectors $\left\{ \varphi_{i}\right\} _{i\in I}\subset\mathcal{H}$
	is called a \emph{frame} for $\mathcal{H}$ if there
	exist positive constants $A\leq B$ such that 
\begin{equation}
	A\norm {f}^{2}_{\mathcal{H}}\leq\sum_{i\in I}\left|\left\langle f,\varphi_{i}\right\rangle \right|^{2}\leq B\norm {f}^{2}_{\mathcal{H}}\label{eq:1}		
\end{equation}
for all vectors $f\in\mathcal{H}$.

We say that $\left\{ \varphi_{i}\right\} _{i\in I}$ is a
\emph{Bessel sequence} in $\mathcal{H}$ if only
the RHS inequality in $\left(\ref{eq:1}\right)$ is satisfied. If $A=B$, $\left\{ \varphi_{i}\right\} _{i\in I}$
is called a \emph{tight frame} and if $A=B=1$ it is called a \emph{Parseval frame.} \par

\end{defn}
\begin{defn}
	A system of vectors $\left\{ \varphi_{i}\right\} _{i\in I}\subset\mathcal{H}$
	is called a \emph{Riesz sequence} in $\mathcal{H}$
	if there exist positive constants $A\leq B$ such that 
	\begin{equation}
		A\sum_{i\in I}\left|a_{i}\right|^{2}\leq\norm{\sum_{i\in I}a_{i}\varphi_{i}}_{\mathcal{H}}^{2}\leq B\sum_{i\in I}\left|a_{i}\right|^{2} \label{eq:2}
	\end{equation}
	for every finite sequence of scalars $\left\{ a_{i}\right\} _{i\in I}$.
\end{defn}

In this paper we discuss Riesz sequences of exponentials,
i.e. for bounded sets $\SS$ of positive Lebesgue measure and a countable
set $\Lambda\subset\RR$ we consider the exponential system $E\left(\Lambda\right)=\left\{ e^{i\lambda x}\right\} _{\lambda\in\Lambda}$
in the space $L^{2}\left(\SS\right)$.

Throughout this paper we will denote by $\left|\SS\right|$ the Lebesgue measure of the set $\SS$.
Note that if $\Lambda$ is a separated set, i.e. $\inf_{\lambda\neq\mu}\left|\lambda-\mu\right|>0$,
then the RHS of $\left(\ref{eq:1}\right)$ and $\left(\ref{eq:2}\right)$ is satisfied for $E(\Lambda)$.
In particular, if $\SS\subset\TT=\RR/2\pi\ZZ$ is a set of positive
Lebesgue measure and $\Lambda\subseteq\ZZ$, then we may take $B=2\pi$.
Also, it follows from Parseval's identity that the normalized exponential system $\frac{1}{\sqrt{2\pi}} E\left(\ZZ\right)$ is a Parseval frame
in $L^{2}\left(\SS\right)$. 

The case where $I$ is an interval is classical. In order
to determine whether a system $E\left(\Lambda\right)$ is a Riesz
sequence in $L^{2}\left(I\right)$, one essentially needs to know
the upper density of $\Lambda$ 
\[
	D^{+}\left(\Lambda\right)=\lim_{r\rightarrow\infty}\sup_{a\in\RR}\frac{\#\left(\Lambda\cap\left(a,a+r\right)\right)}{r}.
\]
\par

\begin{rem*}
	Analogously one may define lower density, and in case these
	quantities coincide we say $\Lambda$ has \emph{uniform density}
	denoted $D\left(\Lambda\right)$.
\end{rem*}

Kahane \cite{kahane1957fonctions} proved the following result, see also {\cite{olevskii2016functions}}.

\begin{thm}[\cite{kahane1957fonctions}] \label{t1.3}
	If $D^{+}\left(\Lambda\right)<\frac{\left|I\right|}{2\pi}$,
	then $E\left(\Lambda\right)$ is a Riesz sequence in $L^{2}\left(I\right)$.
	On the other hand, if $D^{+}\left(\Lambda\right)>\frac{\left|I\right|}{2\pi}$,
	then $E\left(\Lambda\right)$ is not a Riesz sequence in $L^{2}\left(I\right)$.
\end{thm}

For arbitrary sets of positive measure the situation is much
more complicated and only necessary condition exists due to Landau \cite{landau1967necessary}.

\begin{thm}[\cite{landau1967necessary}] \label{t1.4}
	Let $\SS \subset \RR$ be a measurable set. If $E\left(\Lambda\right)$
	is a Riesz sequence in $L^{2}\left(\SS\right)$, then $D^{+}\left(\Lambda\right)\leq\frac{\left|\SS\right|}{2\pi}$.
\end{thm}
In light of Theorem \ref{t1.4} it is natural to ask:

\emph{Given a set $\SS$, does there exist a set $\Lambda$
of positive density such that the exponential system $E\left(\Lambda\right)$
is a Riesz sequence in $L^{2}\left(\SS\right)$?}

This question may be considered under various notions of
density. The first result on this subject was obtained by Bourgain
and Tzafriri \cite[Theorem 2.2]{MR890420} using their celebrated restricted invertibility theorem.

\begin{thm}[{\cite{MR890420}}]\label{t1.5}
	Given $\SS\subset\TT$ of positive measure, there exists
	a set $\Lambda\subset\ZZ$ with positive asymptotic density
	\[
		\text{dens}\left(\Lambda\right)=\lim_{r\rightarrow\infty}\frac{\#\left(\Lambda\cap\left(-r,r\right)\right)}{2r}>c\left|\SS\right|
	\]
	and such that $E\left(\Lambda\right)$ is a Riesz sequence in $L^{2}\left(\SS\right)$.
	Here, $c$ is an absolute constant, independent of $\SS$.
\end{thm}

Theorem \ref{t1.5} implies that every set $\SS$ admits a Riesz
sequence $\Lambda$ with positive upper density, proportional to the
measure of $\SS$, and so we can answer the question above positively.
Another related property is a relatively dense set, also referred to as a syndetic set.

\begin{defn}
	A subset $\Lambda=\left\{ \ldots<\lambda_{0}<\lambda_{1}<\lambda_{2}<\ldots\right\} \subset\mathbb{Z}$
	is called \emph{syndetic} if the consecutive gaps
	in $\Lambda$ are bounded, i.e., if there is a positive constant $d$
	such that 
	\begin{equation}
		\lambda_{n+1}-\lambda_{n}\leq d\quad\forall n\in\ZZ.\label{eq:3}
	\end{equation}
	We denote the smallest possible constant satisfying $\left(\ref{eq:3}\right)$
	by $\gamma\left(\Lambda\right)$.
\end{defn}

The existence of syndetic Riesz sequences has been proved by Lawton \cite[Corollary 2.1]{lawton2010minimal}, and by a different 
approach by Paulsen \cite[Theorem 1.1]{paulsen2011syndetic} to be equivalent to the Feichtinger conjecture.

\begin{thm}[\cite{lawton2010minimal}] \label{t1.7}
	Given a set $\SS\subset\TT$ of positive measure, the following are equivalent: 
	\begin{enumerate}
	\item [(i)]{There exists $r\in\NN$ and  $\ZZ=\bigcup_{j=1}^{r}\Lambda_{j}$
		such that $E\left(\Lambda_{j}\right)$ is a Riesz sequence in $L^{2}\left(\SS\right)$
		for all $j\in \left[r\right] =\{1,\ldots,r\}$.}
	\item [(ii)]{There exists $d\in\NN$ and a syndetic set $\Lambda\subseteq\ZZ$
		with $\gamma\left(\Lambda\right)=d$ such that $E\left(\Lambda\right)$
		is a Riesz sequence in $L^{2}\left(\SS\right)$.} 
	\end{enumerate}
\end{thm}

\begin{rem}
	While we can trivially deduce $r\leq d$ in the implication (ii) $\implies$ (i) by considering translates
	of $\Lambda$, it should be emphasized that the proof of Theorem \ref{t1.7}
	does not imply any connection in the other direction. Explicitly,
	knowing the number of partitions $r$ in $\left(\text{i}\right)$
	does not give any upper bound on $\gamma\left(\Lambda\right)$ in
	$\left(\text{ii}\right)$. 
\end{rem}

The first statement in Theorem \ref{t1.7} is known as the Feichtinger
conjecture for exponentials. The Feichtinger conjecture in its general
form states that every bounded frame can be decomposed into finitely
many Riesz sequences. It has been proved by Casazza et al. in \cite{casazza2005frames} as 
well as by Casazza and Tremain \cite{casazza2006kadison},
to be equivalent to the Kadison-Singer problem. The latter has been
solved recently by Marcus, Spielman and Srivastava. Their main result
is the following theorem \cite[Theorem 1.4]{marcus2013interlacing}.

\begin{thm}[\cite{marcus2013interlacing}] \label{MSS1}
	If $\varepsilon>0$ and $v_{1},\ldots,v_{m}$
	are independent random vectors in $\mathbb{C}^{d}$ with finite support
	such that ${\displaystyle \sum_{i=1}^{m}\mathbb{E}\left[v_{i}v_{i}^{*}\right]= \mathbf{I}_{d}}$
	and $\mathbb{E}\left[\left\Vert v_{i}\right\Vert ^{2}\right]\leq\varepsilon$
	for all $i$, then 
	\[
		\mathbb{P}\left(\left\Vert \sum_{i=1}^{m}v_{i}v_{i}^{*}\right\Vert \leq\left(1+\sqrt{\varepsilon}\right)^{2}\right)>0.
	\] 
\end{thm}

Here and below $\mathbf I_{d}$ is the $d\times d$ identity matrix.  
Later, using similar techniques, Bownik, Casazza, Marcus and
Speegle \cite[Theorem 1.2]{bownik2016improved} gave an improved bound in the case when random vectors $v_{1},\ldots,v_{m}$
each have support of size $2$.

\begin{thm}[\cite{bownik2016improved}] \label{MSS2}
	If $\varepsilon\in\left(0,\nicefrac{1}{2}\right)$
	and $v_{1},\ldots,v_{m}$ are independent random vectors in $\mathbb{C}^{d}$
	with support of size 2 such that ${\displaystyle \sum_{i=1}^{m}\mathbb{E}\left[v_{i}v_{i}^{*}\right]= \mathbf I_{d}}$
	and $\mathbb{E}\left[\left\Vert v_{i}\right\Vert ^{2}\right]\leq\varepsilon$
	for all $i$, then 
	\[
		\mathbb{P}\left(\left\Vert \sum_{i=1}^{m}v_{i}v_{i}^{*}\right\Vert \leq1+2\sqrt{\varepsilon\left(1-\varepsilon\right)}\right)>0.
	\]
\end{thm}

Theorems \ref{MSS1} and \ref{MSS2}  give the following quantitative bounds on the number of partitions $r$, see \cite[Theorem 1.4]{bownik2016improved}.

\begin{thm}[\cite{bownik2016improved}] \label{bcms}
	Let $\eps>0$ and suppose that $\left\{ u_{i}\right\} _{i\in I}$
	is a Bessel sequence in $\mathcal{H}$ with bound $1$ that consists
	of vectors satisfying $\norm{u_{i}}^2\geq\eps$. Then there exists a universal
	constant $C>0$, such that $I$ can be partitioned into $r\leq\frac{C}{\eps}$
	subsets $I_{1},\ldots,I_{r}$ such that every subfamily $\left\{ u_{i}\right\} _{i\in I_{j}}$,
$j=1,\ldots,r$, is a Riesz sequence in $\mathcal{H}.$ Moreover, if $\eps>3/4$, then $r=2$ works.
\end{thm}

Applying Theorem \ref{bcms} to exponential systems they obtained the following corollary \cite[Corollary 6.16]{bownik2016improved}.

\begin{thm}[\cite{bownik2016improved}]
	There exists a universal constant $C>0$ such that for
	any subset $\SS\subset\TT$ with positive measure, the exponential
	system $E\left(\ZZ\right)$ can be decomposed as a union of ${\displaystyle r\leq\nicefrac{C}{\left|\SS\right|}}$
	Riesz sequences $E\left(\Lambda_{j}\right)$ in $L^{2}\left(\SS\right)$
	for $j=1,\ldots,r$. Moreover, if $|\SS| >3/4$, then $r=2$ works.

\end{thm}

\section{Results}\label{S2}

The problem of attaining an explicit bound on the gap in the syndetic
Riesz sequence was initially considered by Casazza and Tremain \cite{casazza2016consequences}.
Their result \cite[Theorem 6.17]{casazza2016consequences} states that every set $\SS$ admits a syndetic Riesz sequence $E\left(\Lambda\right)$, with $\gamma\left(\Lambda\right)\leq c\left|\SS\right|^{-8}$, but no proof was given.

The proof of Theorem \ref{bcms} is based on defining random
vectors appropriately so that with positive probability they facilitate
a partition of the index set $I$ with desired properties. By modifying
the definition of those random vectors we establish a quantitative
bound for $\gamma\left(\Lambda\right)$. Our main result takes the form of a selection theorem for Bessel sequences with norms bounded from below.

\begin{thm}\label{A}
There exists a universal constant $C>0$ such that the following holds.
Let $\eps>0$ and suppose
	that $\left\{ u_{i}\right\} _{i\in I}$ is a Bessel sequence in $\HH$
	with bound $1$ and 
	\[
		\norm{u_{i}}^{2}\geq\eps\qquad\forall i\in I.
	\]
	Then whenever $\left\{ J_{k}\right\} _{k}$
	is a collection of disjoint subsets of $I$ with $\#J_{k}\geq r=\lceil \frac{C}{\eps}\rceil $,
	for all $k$, there exists a selector $J\subset\bigcup_{k}J_{k}$
	satisfying 
	\[
		\#\left(J\cap J_{k}\right)=1\qquad\forall k
	\]
	and such that $\left\{ u_{i}\right\} _{i\in J}$ is a Riesz sequence
	in $\mathcal{H}$. Moreover, if $\eps>\frac{3}{4}$, then the same
	conclusion holds with $r=2$.
\end{thm}

The proof of Theorem \ref{A} is given in Section \ref{S3}. 
Applying Theorem \ref{A} to the normalized exponential system $\{\frac{1}{\sqrt{2\pi}}e^{int}\}_{n\in\ZZ}$ in $L^2(\SS)$ with
 $J_{k}=[rk,r(k+1)) \cap \ZZ $, $k\in\ZZ$, where $r=\lceil {C}/{|\SS|}\rceil$, we obtain the following corollary.
 
\begin{cor}\label{c2.1}
	There exists a universal constant $C>0$ such that for any
	subset $\SS\subset\TT$ with positive measure, there exists a syndetic
	set $\Lambda\subset\ZZ$ with ${\displaystyle \gamma\left(\Lambda\right)\leq C\left|\SS\right|^{-1}}$
	so that $E\left(\Lambda\right)$ is a Riesz sequence in $L^{2}\left(\SS\right)$.
	Moreover, if $\frac{|\SS|}{2\pi}>\nicefrac{3}{4}$, then such $\Lambda$
	exists with $\gamma\left(\Lambda\right)\leq3$.
\end{cor}

\begin{rem*}
	Beyond improvements to the constant $C$, the bound in Corollary
	\ref{c2.1} is asymptotically sharp as $\left|\SS\right|\rightarrow0$. That is, we cannot expect to find a syndetic Riesz sequence with gaps of the order $\left|\SS\right|^{\alpha}$ as $|\SS| \to 0$ with $\alpha>-1$.	
	This
	follows from Theorem \ref{t1.4} due to density considerations of the set
	$\Lambda$. 
\end{rem*}
	
Theorem \ref{A} can also be applied to the multidimensional setting. In this case we can deduce the existence of exponential Riesz sequences with syndetic-like properties.
\begin{cor}\label{rect}
	There exists a universal constant $C>0$ such that for any subset $\SS\subset\TT^{d}$ 
	of positive measure, any $d$-dimensional rectangle $\mathcal{R}$ with 
	$\# \mathcal{R} \geq C |\SS|^{-1}$, and any partition $\ZZ^{d}=\bigcup {\mathcal{R}}_k$ into disjoint union of translated copies of $\mathcal{R}$, there exists a set $\Lambda\subset\ZZ^{d}$ such that
	\[
		\# \left(\Lambda\cap\mathcal{R}_k\right)=1\qquad\forall k
	\]
	and $E\left(\Lambda\right)$ is a Riesz sequence in $L^{2}\left(\SS\right)$.
	If $\mathcal{R}$ is taken to be a cube, then we obtain a set $\Lambda\subset\ZZ^{d}$ with
	\begin{equation}\label{cub}
		\sup_{\mu\in\mathbb Z^d } \inf_{\lambda\in\Lambda}\left|\lambda-\mu\right|\leq C\sqrt{d}\left|\SS\right|^{-\nicefrac{1}{d}}. 
	\end{equation}
\end{cor}

Note that Theorem \ref{A} does not put any restrictions on the choice of sets $\{J_k\}_k$ beyond the size requirement $\# J_k \ge r$ for all $k$. Consequently, in Corollary \ref{rect} we may choose for $\{\mathcal R_k\}_{k}$ any family of disjoint subsets of $\ZZ^d$ satisfying $\# \mathcal R_k \ge r$, where $r=C/|\SS|$. However, the most interesting case occurs when we take a cube $\mathcal R= [0,s)^d \cap \ZZ^d$ with side length $s=\lceil (C/|\SS|)^{-1/d} \rceil$ and the corresponding lattice partition
\[
\mathcal R_k = k + \mathcal R, \qquad k\in s\ZZ^d.
\]
This choice yields a set $\Lambda \subset \ZZ^d$ satisfying the bound \eqref{cub} such that $E\left(\Lambda\right)$ is a Riesz sequence in $L^{2}\left(\SS\right)$.
Partitioning the lattice $\ZZ^d$ in a more complicated pattern we obtain the following corollary.

\begin{cor}\label{c2.4}
	There exists a universal constant $C>0$ such that for any subset $\SS\subset\TT^{d}$ 
	of positive measure, there exists a set $\Lambda \subset \ZZ^{d}$ so that
	$E\left(\Lambda\right)$ is a Riesz sequence in $L^{2}\left(\SS\right)$ and $\Lambda$ is syndetic along any 
	of its one dimensional sections. 
	That is, for any $j=1,\ldots,d$, and for 
	any $(k_{1},\ldots,\hat k_j, \ldots, k_{d})\in\ZZ^{d-1}$ the set
	\[
		\Lambda_j(k_{1},\ldots,\hat k_j, \ldots, k_{d})=\{k_j\in \ZZ :(k_{1},\ldots,k_j, \ldots,  k_{d} ) \in \Lambda\}
	\]
is a syndetic subset of integers with gap satisfying 
	\begin{equation}\label{gap}
		\gamma (\Lambda_j (k_{1},\ldots,\hat k_j, \ldots, k_{d}) ) \leq Cd |\SS |^{-1}.
	\end{equation}
\end{cor}

	The notation $\hat k_j$ in the vector $(k_{1},\ldots,\hat k_j, \ldots, k_{d})$ means that the term $k_j$ is missing.
\begin{proof}
	Let $\SS\subset\TT^{d}$ of positive measure. We wish to apply Theorem \ref{A} to the normalized exponential system 
	$(2\pi)^{-d/2} E\left(\ZZ^{d}\right)$, which is a Parseval frame in $L^{2}\left(\SS\right)$, so we are only left  
	to specify how $\{J_{k}\}$ should be chosen.
	
Let $r= \lceil C/|\SS| \rceil$, where $C$ is the same constant as in Theorem \ref{A}. Define the $d$-dimensional cube 
\[
Q_{r,d}=\{0,\ldots,r-1\}^{d}.
\]
We shall construct a partition $\{J_{k,x}\}_{\left(k,x\right) \in (r\ZZ)^d \times Q_{r,d-1}}$ of $\ZZ^d$ as follows. For $k\in (r\ZZ)^d$, let $j=j(k) \in\{ 1,\ldots, d\} $  be such that
\begin{equation}\label{cyc}
	j \equiv \frac{k_1+ \ldots+k_d}r \mod d.
\end{equation}
Define 
\[
J_{k,x} = \{(k_1+x_{1},\ldots,k_{j-1}+x_{j-1},k_j+x,
		k_{j+1}+x_{j},\ldots, k_{d}+x_{d-1}) : x\in\{0,\ldots,r-1\}\}.
\]
Hence, for a fixed $k\in (r\ZZ)^d$, the family $\{J_{k,x}\}_{x\in  Q_{r,d-1}}$ is a partition of a cube $k+ Q_{r,d}$ into line segments of length $r$ parallel to the $j$ axis. Now, the direction $j$ cycles according to \eqref{cyc} as a function of $k\in (r\ZZ)^d$. In particular, for a fixed $j_0 =1,\ldots, d$, cubes of the form $k+r(nd+j_0-j) e_{j_0}+Q_{r,d}$, where $n\in \ZZ$, $j$ is given by $(\ref{cyc})$ and $e_{j_0}$ is the $j_0$'th coordinate vector, are partitioned into line segments parallel to the $j_0$ axis. Since Theorem \ref{A} guarantees that the selection set $\Lambda$ contains one element from each such line segment, we deduce that $\Lambda$ is syndetic along every one dimensional section in the direction of $e_{j}$ with gap bounded by $2dr$. This yields the bound \eqref{gap}. Moreover, $E(\Lambda)$ is a Riesz sequence in $L^2(\SS)$ by Theorem \ref{A}.
\end{proof}

In Section \ref{S3} we discuss the proof of Theorem \ref{A} in detail.
Its proof uses Theorems \ref{MSS1} and \ref{MSS2}, which involve probabilistic elements. This, in essence, makes the task of extracting a deterministic 
construction of the set $\Lambda$ extremely difficult,
and the authors do not know how this might be done for an arbitrary
set. Furthermore, while Corollary \ref{c2.1} is sharp in the sense that asymptotically as $|\SS| \to 0$, we are also
interested in another kind of sharpness which may occur as $\frac{|\SS|}{2\pi} \rightarrow1$. That is, for sets of large measure we wish to remove from $\ZZ$ a uniformly separated set with gaps size not smaller than $C\left|\SS^{c}\right|^{-1}$.
It turns out that in the case where $\SS$ is an open set a deterministic
construction does exist and can be chosen to satisfy these properties.

\begin{thm}\label{B}
	There exists a
	universal constant $C>0$ such that for any open subset $\SS\subset\TT$,
	one can construct a syndetic set $\Lambda\subset\ZZ$ with gaps between consecutive elements
	taking exactly two values $\{1,d\}$, where $d=d\left(\SS\right) \leq \frac{C}{|\SS|}$,
	and so that $E\left(\Lambda\right)$ is a Riesz sequence in $L^{2}\left(\SS\right)$.
	Moreover, if $\frac{|\SS|}{2\pi}>\frac{1}{2}$, then $\Lambda$ may be
	chosen so that $\Lambda^{c}=\ZZ\backslash\Lambda$ is uniformly separated and satisfies 
\begin{equation}\label{eq:7}
		\inf_{\lambda,\mu\in\Lambda^{c},\lambda\neq\mu}\left|\lambda-\mu\right|\geq\frac{C}{\left|\SS^{c}\right|}.
\end{equation}
\end{thm}

	It should be noted that the sheer existence of syndetic Riesz sequences for open sets is trivial since any open set admits an infinite arithmetic progression as a Riesz sequence. However, Theorem \ref{B} ensures the existence of such systems together with an effective estimate on the gap.
The proof of Theorem \ref{B} is based on a totally different approach than that of Theorem \ref{A}
and uses simple quasicrystals. This will be discussed in Section \ref{S4}.

\section{The syndetic riesz sequence bound - general case} \label{S3}

In this section we prove our main result Theorem \ref{A}.
Theorem \ref{MSS1} gives rise to a selecting mechanism which allows,
from every Parseval frame with certain norms, a probabilistic selection
of a Bessel subsystem with smaller bound. Note that by the Schur-Horn
theorem we may apply Theorem \ref{MSS1} to Bessel sequences with bound $1$ by completing them to a Parseval frame. Indeed, we have the following lemma from \cite{bownik2016improved}.

\begin{lem}\label{l3.1}
	Let $\HH$ be an infinite dimensional Hilbert space, $M\in\NN$ and $\delta\in\left(0,1\right)$. Suppose
	$\left\{ u_{i}\right\} _{i=1}^{M}\subset\HH$ is a Bessel sequence
	with Bessel bound $1$ and $\norm{u_{i}}^{2}\geq\delta$ for all $i$,  or $\norm{u_{i}}^{2}\leq\delta$ for all $i$.
	Then for every large enough $K\in\NN$, there exist vectors $\varphi_{1},\ldots,\varphi_{K}\in\mathcal{H}$
	with $\norm{\varphi_{i}}^{2}\geq\delta$ for all $i$, or $\norm{\varphi_{i}}^{2}\leq\delta$ for all $i$, respectively, such that $\left\{ u_{i}\right\} _{i=1}^{M}\cup\left\{ \varphi_{i}\right\} _{i=1}^{K}$
	is a Parseval frame for its linear span.
\end{lem}

\begin{proof} The case $\norm{u_{i}}^{2}\geq\delta$ of Lemma \ref{l3.1} follows from the Schur-Horn theorem
and it can be found in \cite[Corollary 6.6]{bownik2016improved}. 
However, the case $\norm{u_{i}}^{2}\leq\delta$ is much simpler as it  does not
	require the use of the Schur-Horn theorem. Indeed, consider the frame operator $S$ which corresponds to
	$\left\{ u_{i}\right\} _{i=1}^{M}$. It is well known that $S$ is a self-adjoint 
	operator (see \cite{olevskii2016functions}), and so denoting its non-zero eigenvalues
	$1 \ge \lambda_{n}\geq\ldots\geq\lambda_{1} >0$ and corresponding orthonormal
	eigenvectors $v_{1},\ldots,v_{n}$ we may write 
	\[
		S=\sum_{i=1}^{n}\lambda_{i}\left\langle \cdot,v_{i}\right\rangle v_{i}.
	\]
	Then we choose a number $m$ such that $\frac{1-\lambda_{n}}{m}<\delta$,
	and for each $i\in\left[n\right]$ we take $m$ copies of the vector
	$\sqrt{\frac{1-\lambda_{i}}{m}}v_{i}$. By the choice of $m$ 
	\[
		\norm{\sqrt{\frac{1-\lambda_{i}}{m}}v_{i}}^{2}\leq\delta\qquad\forall i\in\left[n\right]
	\]
	and 
	\[
		S+\sum_{j=1}^{m}\sum_{i=1}^{n}\frac{1-\lambda_{i}}{m}\left\langle \cdot,v_{i}\right\rangle v_{i}={\mathbf I}_{H},
	\]
where $H=\spa\{u_1,\ldots,u_M\} = \spa\{v_1,\ldots,v_n\}$. Here and below $\mathbf I_{\mathcal{H}}$ denotes the identity map on $\mathcal{H}$. 
\end{proof}

Recall that by Naimark's dilation theorem \cite[Proposition 1.1]{han2000frames} a system
$\left\{ u_{i}\right\}_{i\in I}$ is a Parseval frame for $H$ if and only if
there exists a containing space $H\mathcal{\subseteq H}$ and an orthonormal
basis $\left\{ e_{i}\right\}_{i\in I} $ for $\mathcal{H}$ such that $Pe_{i}=u_{i}$
for all $i$. Here $P:\mathcal{\HH\rightarrow}\HH$
is the orthogonal projection onto $H$. It turns out that by considering
a certain dual system, called Naimark's complement, we may choose
a subsystem of $\left\{ u_{i}\right\}_{i\in I}$ which is a Riesz sequence, see 
\cite[Proposition 5.4]{bownik2016improved}.

\begin{lem}[\cite{bownik2016improved}] \label{l3.3}
	Let $P:\HH\rightarrow\HH$ be the orthogonal projection
	onto a closed subspace $H\subset\HH$, and let $\left\{ e_{i}\right\} _{i\in I}$
	be an orthonormal basis for $\HH$. Then for any subset $J\subset I$
	and $\delta>0$ the following are equivalent: 
	\begin{enumerate}
	\item {$\left\{ Pe_{i}\right\} _{i\in J}$ is a Bessel sequence
		with bound $1-\delta$.} 
	\item {$\left\{ \left(\mathbf I_{\HH}-P\right)e_{i}\right\} _{i\in J}$
		is a Riesz sequence with lower Riesz bound $\delta$.}
	\end{enumerate}
\end{lem}

The following theorem demonstrates how Theorem \ref{MSS1} can be applied in order to find a selector.

\begin{thm}\label{t3.4}
	Let $r,M\in\NN$ and $\delta>0$. Suppose that $\left\{ u_{i}\right\} _{i=1}^{M}\subset\HH$
	is a Bessel sequence with bound $1$ and $\norm{u_{i}}^{2}\leq\delta$
	for all $i$. Then for every collection of disjoint subsets $J_{1},\ldots,J_{n}\subset\left[M\right]$
	with $\#J_{k}\ge r$ for all $k$, there exists a subset $J\subset\left[M\right]$
	such that 
	\begin{equation}\label{sel}
	\#\left(J\cap J_{k}\right)=1 \qquad \forall k\in\left[n\right]
	\end{equation}
	and the system of vectors $\left\{ u_{i}\right\} _{i\in J}$ is a
	Bessel sequence with bound
\[
	\left(\frac{1}{\sqrt{r}}+\sqrt{\delta}\right)^{2}.
\]
\end{thm}

\begin{proof}
	Without loss of generality we can assume that $\#J_{k}= r$ for all $k$. Consider the
	system of vectors $\left\{ u_{i}\right\} _{i\in\bigcup_{k=1}^{n}J_{k}}$
	and note that it is Bessel sequence with bound $1$. By Lemma \ref{l3.1}
	we can find finitely many vectors $\varphi_{1},\ldots,\varphi_{rK}$
	with $\norm{\varphi_{i}}^{2}\leq\delta$ for all $i$ and such that
	$\left\{ u_{i}\right\} _{i\in\bigcup_{k=1}^{n}J_{k}}\cup\left\{ \varphi_{i}\right\} _{i=1}^{rK}$
	is a Parseval frame for some finite dimensional subspace $H\subset\HH$. \par

	Define the independent random vectors $v_{1}, \ldots, v_{n+K}$ as follows. For $k=1,\ldots,n$ we let the
	vector $v_{k}$ to take values $\sqrt{r}u_{i}$
	for any $i\in J_{k}$ with equal probability $\frac{1}{r}$.
	For $k=n+1,\ldots n+K$, we set
\[
	\mathbb{P}\left(v_{k}=\sqrt{r}\varphi_{r\left(k-\left(n+1\right)\right)+1}\right)=\nicefrac{1}{r}\,,\ldots,\,\mathbb{P}\left(v_{k}=\sqrt{r}\varphi_{r\left(k-n\right)}\right)=\nicefrac{1}{r}
\]

Note that
\[
	\mathbb{E}\norm{v_{k}}^{2}\leq r\delta\qquad\forall k\in\left[n+K\right]
\]
 as well as 
\[
	\sum_{k=1}^{n+K}\mathbb{E}\left(v_{k}v_{k}^{*}\right)=\sum_{i\in\bigcup_{k=1}^{n}J_{k}}u_{i}u_{i}^{*}+\sum_{i=1}^{rK}\varphi_{i}\varphi_{i}^{*}=\mathbf {I}_{H}.
\]
It follows from Theorem \ref{MSS1} that
\[
	\mathbb{P}\left(\left\Vert \sum_{k=1}^{n+K}v_{k}v_{k}^{*}\right\Vert \leq\left(1+\sqrt{r\delta}\right)^{2}\right)>0.
\]
This implies the existence of a set $J\subset\left[M\right]$ satisfying \eqref{sel} such
that 
\[
	\left\Vert \sum_{i\in J}u_{i}u_{i}^{*}\right\Vert \leq\left(\frac{1}{\sqrt{r}}+\sqrt{\delta}\right)^{2}
\]
as required.
\end{proof}

In the case the vectors in our systems are short we can obtain a variant of Theorem \ref{t3.4}. Theorem \ref{t3.5} follows from Theorem \ref{MSS2} in a similar fashion as Theorem \ref{t3.4} follows from Theorem \ref{MSS1}, see \cite[Theorem 6.3]{bownik2016improved}.
 
\begin{thm}\label{t3.5}
	Let $M\in\NN$ and $\delta_{0}\in\left(0,\nicefrac{1}{4}\right)$.
	Suppose that $\left\{ u_{i}\right\} _{i=1}^{M}\subset\mathcal{H}$
	is a Bessel sequence with Bessel bound $1$ and $\norm{u_{i}}^{2}\leq\delta_{0}$
	for all $i$. Then for every collection of disjoint subsets $J_{1},\ldots,J_{n}\subset\left[M\right]$
	with $\#J_{k}=2$ for all $k$, there exists a subset $J\subset\left[M\right]$
	such that $\#\left(J\cap J_{k}\right)=1$ for all $k\in\left[n\right]$, 
	and the system of vectors $\left\{ u_{i}\right\} _{i\in J}$ is a
	Bessel sequence with bound $1-\eps_{0}$, where $\eps_{0}=\frac{1}{2}-\sqrt{2\delta_{0}\left(1-2\delta_{0}\right)}.$ \par
\end{thm}

Throughout this section we fix some value of $\delta_{0} \in (0,\nicefrac{1}{4})$ and the corresponding $\eps_{0}=\frac{1}{2}-\sqrt{2\delta_{0}\left(1-2\delta_{0}\right)}$.
As a corollary we obtain a way to extract a syndetic Riesz sequence from a given Bessel sequence, provided that the vectors in the system are not too short.

\begin{cor}\label{c3.6}
	Let $M\in\NN$ and $\delta_{0}\in\left(0,\nicefrac{1}{4}\right)$.
	Suppose that $\left\{ u_{i}\right\} _{i=1}^{M}\subset\mathcal{H}$
	is a Bessel sequence with Bessel bound $B$ and $\norm{u_{i}}^{2}\geq B\left(1-\delta_{0}\right)$
	for all $i$. Then for every collection of disjoint subsets $J_{1},\ldots,J_{n}\subset\left[M\right]$
	with $\#J_{k}=2$ for all $k$, there exists a subset $J\subset\left[M\right]$
	such that $\#\left(J\cap J_{k}\right)=1$ for all $k\in\left[n\right]$
	and the system of vectors $\left\{ u_{i}\right\} _{i\in J}$ is a
	Riesz sequence with lower Riesz bound $B\eps_{0}$.
\end{cor}

\begin{proof}	
	Without loss of generality we may assume $B=1$. Let $J_{1},\ldots,J_{n}\subset\left[M\right]$
	be a disjoint collection of subsets with $\#J_{k}=2$ for all $k$.
	Consider the system of vectors $\left\{ u_{i}\right\} _{i\in\bigcup_{k=1}^{n}J_{k}}$
	and note that it is Bessel sequence with bound $1$. Then applying
	Lemma \ref{l3.1} to $\left\{ u_{i}\right\} _{i\in\bigcup_{k=1}^{n}J_{k}}$,
	we find finitely many vectors $\varphi_{1},\ldots,\varphi_{2K}$ with
	$\norm{\varphi_{i}}^{2}\geq1-\delta_{0}$ for all $i$ and such that
	$\left\{ u_{i}\right\} _{i\in\bigcup_{k=1}^{n}J_{k}}\cup\left\{ \varphi_{i}\right\} _{i=1}^{2K}$
	is a Parseval frame for its linear span, denoted by $H$. \par

	For convenience we rename our vectors as $\left\{ g_{i}\right\} _{i=1}^{2\left(n+k\right)}$
	as follows. We write $J_{k}=\left\{ j_{1}^{k},j_{2}^{k}\right\} $
	and define 
	\[
		g_{2k-1}=u_{j_{1}^{k}},\quad g_{2k}=u_{j_{2}^{k}},\qquad k=1,\ldots,n,
	\]
	and $\left\{ g_{i}\right\} _{i=2n+1}^{2\left(n+K\right)}=\left\{ \varphi_{i}\right\} _{i=1}^{2K}$.
	By Naimark's dilation theorem there exists a $2\left(n+K\right)$
	dimensional space $\mathcal{K}\supseteq H$ with a corresponding orthonormal
	basis $\left\{ e_{i}\right\} _{i=1}^{2\left(n+K\right)}$ and an orthogonal
	projection $P:\mathcal{K}\rightarrow \mathcal{K}$ onto $H$ such that 
	\[
		Pe_{i}=g_{i}\qquad\text{for }i=1,2,\ldots,2\left(n+K\right)
	\]
	Then, setting $f_{i}=\left(\mathbf I_{\mathcal{K}}-P\right)e_{i}$ for $i=1,\ldots,2\left(n+K\right)$, 
	it follows that $\left\{ f_{i}\right\} $ is a Parseval frame
	for its linear span with $\norm{f_{i}}^{2}\leq\delta_{0}$. By Theorem
	\ref{t3.5} applied to the collection $J_{k}^{\prime}=\left\{ 2k-1,2k\right\} $,
	$k=1,\ldots,n+K$, there is a subset $J^{\prime}\subset\left[2\left(n+K\right)\right]$
	such that $\#\left(J^{\prime}\cap J_{k}^{\prime}\right)=1$ for all
	$k\in\left[n+K\right]$ and $\left\{ f_{i}\right\} _{i\in J^{\prime}}$
	is a Bessel sequence with bound $1-\eps_{0}$. By Lemma \ref{l3.3} $\left\{ g_{i}\right\} _{i\in J^{\prime}}$
	is a Riesz sequence with lower bound $\eps_{0}$. Hence, there exists
	$J\subset\left[M\right]$ such that $\#\left(J\cap J_{k}\right)=1$
	for all $k\in\left[n\right]$ and the system of vectors $\left\{ u_{i}\right\} _{i\in J}$
	is a Riesz sequence with lower bound $\eps_{0}$. 
\end{proof}

Combining Theorem \ref{t3.4} and Corollary \ref{c3.6} we get a finite dimensional version of Theorem \ref{A}.

\begin{thm}\label{t3.7}
	Let $\eps>0$ and $M\in\NN$. Suppose that
	$\left\{ u_{i}\right\} _{i=1}^{M}\subset\mathcal{H}$ is a Bessel
	sequence with Bessel bound $1$ and $\norm{u_{i}}^{2}\geq\eps$ for
	all $i$. Then there exists $r=O\left(\nicefrac{1}{\eps}\right)$,
	independent of $M,$ such that for every collection of disjoint subsets
	$J_{1},\ldots,J_{n}\subset\left[M\right]$ with $\#J_{k}\geq r$ for
	all $k$, there exists a subset $J\subset\left[M\right]$ such that
	$\#\left(J\cap J_{k}\right)=1$ for all $k\in\left[n\right]$ and
	 $\left\{ u_{i}\right\} _{i\in J}$ is a Riesz
	sequence with lower Riesz bound $\eps\eps_{0}$. Moreover, if $\eps>\nicefrac{3}{4}$ 
	then the same conclusion holds with $r=2$.
\end{thm}

\begin{proof}
	Without loss of generality we may assume $\norm{u_{i}}^{2}=\eps$
	for all $i$; otherwise replace $u_{i}$ with $\nicefrac{\sqrt{\eps}}{\norm{u_{i}}}u_{i}$.
	Let $C$ be an absolute constant to be specified later and
	set $r=2\left\lceil \frac{C}{\eps}\right\rceil $. Let $J_{1},\ldots,J_{n}\subset\left[M\right]$
	be a collection of disjoint subsets with $\#J_{k}\geq r$ for all
	$k$. \par

	Let $\tilde{r}=\frac{r}{2}$. For every $k\in\left[n\right]$ choose a
	partition $J_{k}=\Lambda_{2k-1}\cup\Lambda_{2k}$ such that $\#\Lambda_{\ell}\geq\tilde{r}$
	for all $\ell\in \left[2n\right]$. By Theorem \ref{t3.4} we may find a subset $J\subset\left[2n\right]$
	such that $\#\left(J\cap\Lambda_{\ell}\right)=1$ for all $\ell\in\left[2n\right]$, 
	and the system of vectors $\left\{ u_{i}\right\} _{i\in J}$ is a
	Bessel sequence with bound
	\[
		\left(\frac{1}{\sqrt{\tilde{r}}}+\sqrt{\eps}\right)^{2}.
	\]
	Then we partition $J=\bigcup_{m=1}^{n}P_{m}$ into pairs $P_{m}=\left\{ j_{2m-1},j_{2m}\right\} $
	so that $j_{\ell}\in J\cap \Lambda_{\ell}$ for all $\ell\in \left[2n\right]$. Next we wish
	to apply Corollary \ref{c3.6} to $\left\{ u_{i}\right\} _{i\in J}$, and we
	may do that if for some $\delta_{0}\in\left(0,\nicefrac{1}{4}\right)$
	we have 
	\[
		\eps\geq\left(\frac{1}{\sqrt{\tilde{r}}}+\sqrt{\eps}\right)^{2}\left(1-\delta_{0}\right).
	\]
	By a simple calculation this holds true whenever $\tilde{r}\geq\frac{C}{\eps}$,
	with $C=9\left(\frac{1-\delta_{0}}{\delta_{0}}\right)^{2}$, see the proof \cite[Theorem 6.7]{bownik2016improved}. By Corollary 	\ref{c3.6} applied to the collection $P_{1},\ldots,P_{n}$,
	there exists a subset $J^{\prime}\subset\left[n\right]$ such that
	$\#\left(J^{\prime}\cap P_{m}\right)=1$ for all $m\in\left[n\right]$, 
	and the system of vectors $\left\{ u_{i}\right\} _{i\in J^{\prime}}$
	is a Riesz sequence with lower bound $\left(\frac{1}{\sqrt{\tilde{r}}}+\sqrt{\eps}\right)^{2}\eps_{0}\geq\eps\eps_{0}$.
	For the moreover part we apply Corollary \ref{c3.6} directly, setting
	$\delta_{0}=1-\eps$.
\end{proof}

In order to prove Theorem \ref{A} we will require the following lemma.  

\begin{lem}\label{l3.8}
	Let $\left\{ J_{k}\right\} _{k}$ be a collection of disjoint
	subsets of $I$. Assume for every $n\in\NN$ we have a subset $I_{n}\subset\bigcup_{k=1}^{n}J_{k}$
	such that 
	\begin{equation}
		\#\left(I_{n}\cap J_{k}\right)=1\quad \forall k\in\left[n\right].	\label{eq:4}
	\end{equation}
	Then, there exists a subset $I_{\infty}\subset I$ and an increasing
	sequence $\left\{ n_{j}\right\} $ such that 
	\begin{equation}
		I_{n_{j}}\cap\left(\bigcup_{k=1}^{j}J_{k}\right)=I_{\infty}\cap\left(\bigcup_{k=1}^{j}J_{k}\right)\quad \forall j.\label{eq:5}
	\end{equation}
	In particular, we have
	\[
		\#\left(I_{\infty}\cap J_{k}\right)=1\hfill\qquad\forall k\in\NN.
	\]
\end{lem}

Lemma \ref{l3.8} is proved by a combination of a diagonal argument
with the pigeonhole principle. For a similar result involving partitions, see \cite[Proposition 2.1]{casazza2005frames}.

\begin{proof}[Proof of Theorem \ref{A}] Let $r=r\left(\eps\right)$
be as in Theorem \ref{t3.7} and $\left\{ J_{k}\right\} _{k}$ a collection
of disjoint subsets of $I$ with $\#J_{k}\geq r$, for all $k$. \par

For every $n\in\NN$ apply Theorem \ref{t3.7} to the collection
$J_{1},\ldots,J_{n}$. This gives a subset $I_{n}\subset\bigcup_{k=1}^{n}J_{k}$
satisfying $\left(\ref{eq:4}\right)$ and such that the system of vectors $\left\{ u_{i}\right\} _{i\in I_{n}}$
is a Riesz sequence in $\mathcal{H}$ with lower bound $\eps\eps_{0}$.
By Lemma \ref{l3.8} there exist an increasing sequence $\left\{ n_{j}\right\} $
and a subset $I_{\infty}\subset I$ satisfying $\left(\ref{eq:5}\right)$.
We claim that $\left\{ u_{i}\right\} _{i\in I_{\infty}}$ is a Riesz sequence
in $\mathcal{H}$ with the same lower bound. This is true since every
finite subsystem is contained in $\left\{ u_{i}\right\} _{i\in I_{n_{j}}}$,
given that $j$ is large enough.
\end{proof}

It is worth to mention an extension of Theorem \ref{A} in which we are interested in finding a syndetic Riesz sequence with tight bounds. This problem was known as the $R_{\eps}$ conjecture.

\begin{thm}\label{t3.9}
There exists a universal constant $C>0$ such that following holds.	Let $\left\{ u_{i}\right\} _{i\in I}$ be a unit norm Bessel sequence in $\HH$
	with bound $B$. For any $\eps>0$ and any collection of disjoint subsets of $I$, $\left\{ J_{k}\right\} _{k}$ satisfying $\#J_{k}\geq r=\left\lceil C\frac{B}{\eps^{4}}\right\rceil $
	for all $k$, there exists a selector $J\subset\bigcup_{k}J_{k}$ satisfying 
	\[
		\#\left(J\cap J_{k}\right)=1\qquad\forall k
	\]
	and such that $\left\{ u_{i}\right\} _{i\in J}$ is a Riesz sequence
	in $\mathcal{H}$ with bounds $1-\eps$ and $1+\eps$.
\end{thm}

\begin{rem*}
	Theorem \ref{t3.9}, when applied to exponential systems, provides an improvement of \cite[Theorem 6.4]{casazza2016consequences}. A recent multi-paving result of Ravichandran and Srivastava \cite{RS} suggests that this result might hold with $r=O(\frac{B}{\eps^{2}})$.
\end{rem*}

We will give an outline of the proof of Theorem \ref{t3.9}. This will require the following lemma, which appeared in \cite{MR1836633}. 
A slightly more general formulation, given below, is taken from \cite[Lemma 6.13]{bownik2016improved}.

\begin{lem}[\cite{bownik2016improved}] \label{l3.10}
	Suppose $\{u_{i}\}_{i\in I}$ is a Riesz basis in $\mathcal{H}$, and let $\{u^{*}_{i}\}_{i\in I}$ 
	be its unique biorthogonal Riesz basis. Then for any subset $J\subset I$, the Riesz sequence
	bounds of $\{u_{i}\}_{i\in J}$ are A and B if and only if the Riesz sequence bounds of
	$\{u^{*}_{i}\}_{i\in J}$ are $\nicefrac{1}{B}$ and $\nicefrac{1}{A}$. 
\end{lem}

The proof of Theorem \ref{t3.9} proceeds as follows. By Lemma \ref{l3.8} it suffices to consider a
finite system $\left\{ u_{i}\right\}^{M}_{i=1}$ that is a unit norm Bessel sequence with upper bound $B$.
Fix a collection $J_{1},\ldots,J_{n}\subset\left[M\right]$ with $\#J_{k}\geq r$ for
all $k$. Then the subsystem we are looking for is obtained in 3 steps. In the first step, using Theorem \ref{t3.7}, 
we find a subsystem $\left\{u_{i}\right\}_{i\in J}$ which is a Riesz sequence 
with lower Riesz bound $\eps_{0}$ and upper Riesz bound $\frac{1}{1-\delta_{0}}$,
and $\#\left(J\cap J_{k}\right) \geq \lceil \frac{C}{\eps^{2}}\rceil^2$ for all $k$.
In the second step we apply Theorem \ref{t3.4} to $\left\{u_{i}\right\}_{i\in J}$ and extract from it 
another subsystem $\{u_{i}\}_{i\in J^{\prime}}$, $J^{\prime}\subset J$, having its upper Riesz bound 
reduced to $1+\eps$. As a consequence we get $\#\left(J^{\prime} \cap J_{k}\right)  \geq \left\lceil \frac{C}{\eps^{2}}\right\rceil$
for all $k$.
In the last step we move to the dual system $\left\{u^{*}_{i}\right\}_{i\in J^{\prime}}$ 
which, by Lemma \ref{l3.10}, has upper Riesz bound $\frac{1}{\eps_{0}}$. Apply to it Theorem \ref{t3.4} and get a subsystem
$\{u^{*}_{i}\}_{i\in J^{\prime\prime}}$, $J^{\prime\prime}\subset J^{\prime}$ with upper 
Riesz bound reduced to $1+\eps$, and $\#\left(J^{\prime\prime}\cap J_{k}\right)=1$ for all $k$.
Now $\left\{u_{i}\right\}_{i\in J^{\prime\prime}}$ is the desired subsystem, 
keeping in mind that $\frac{1}{1+\eps}\geq 1-\eps$.

\section{Open sets} \label{S4}

In this section we prove Theorem \ref{B}, namely we present
a construction of a syndetic exponential Riesz sequence which is asymptotically sharp in both senses
described in Section \ref{S2}. This will be done applying a particular case
of Fourier quasicrystals. We begin with some background.

\subsection{Universal Riesz sequences}

In their paper \cite{MR2439002} Olevskii and Ulanovskii asked
whether there exists a set $\Lambda$ such that $E\left(\Lambda\right)$
is a Riesz sequence in $L^{2}\left(\SS\right)$ for all sets $\SS\subset\TT$
with large measure. As it turns out, the answer depends considerably on the spectrum $\SS$. It is positive if we restrict ourselves to open spectra.
\begin{thm}[{\cite[Theorem 5]{MR2439002}}] \label{t4.2}
	For every $d\in\left(0,1\right)$, there is a universal
	Riesz sequence, i.e. a set $\Lambda\subset\RR$ with $D\left(\Lambda\right)=d$
	such that $E\left(\Lambda\right)$ is a Riesz sequence in $L^{2}\left(\SS\right)$
	for every open set $\SS\subset\TT$ with $\frac{\left|\SS\right|}{2\pi}>d$.
\end{thm}

Note that the set $\Lambda$ constructed in the proof of
Theorem 4.1 is a specially chosen perturbation of the lattice $\nicefrac{1}{d}\ZZ$. 
On the other hand, considering all measurable sets the answer is negative.

\begin{thm}[{\cite[Theorem 4]{MR2439002}}] \label{t4.1}
	Let $d\in\left(0,1\right)$. Then for every $\eps\in\left(0,1\right)$
	and $\Lambda\subset\RR$ with $D\left(\Lambda\right)=d$, there exists
	a set $\SS\subset\TT$ with $\frac{\left|\SS\right|}{2\pi}>1-\eps$
	such that $E\left(\Lambda\right)$ is not a Riesz sequence in $L^{2}\left(\SS\right)$.
\end{thm}

\subsection{Quasicrystals}

The mathematical model of quasicrystals, introduced by Meyer
in 1972 (see \cite{meyer1972algebraic}), is a construction based on the ``cut and project''
process which produces what is known as Meyer's ``model set''. In
the periodic case a model set takes the following form. We
use the notation from \cite{kozma2011exponential}.

\begin{defn}
	Let $\alpha$ be an irrational number, and $I=\left[a,b\right)\subset \left[0,1\right]$.
	The \emph{(simple) quasicrystal} corresponding to
	$\alpha$ and $I$ is 
	\[
		\Lambda\left(\alpha,I\right)=\left\{ n\in\ZZ\,|\,\left\{ \alpha n\right\} \in I\right\} 
	\]
	where $\left\{ x\right\} $ is the fractional part of the real number $x$.
\end{defn}

The set $\Lambda\left(\alpha,I\right)$ has many interesting
and important properties, see, for instance, \cite{meyer1972algebraic} and \cite{meyer2006nombres}. A basic property which follows from the fact that the sequence
$\left\{ \alpha n\right\} $ is equidistributed in $\left[0,1\right]$ is that the corresponding simple quasicrystal had uniform density $D\left(\Lambda\left(\alpha,I\right)\right)=\left|I\right|$.
One key property of the exponential system $E\left(\Lambda\left(\alpha,I\right)\right)$
that will be useful for us is universality. Matei and Meyer showed that simple quasicrystals are an alternative example of universal Riesz sequences with the additional property that they are contained in the integers rather than in the real line (see \cite{matei2008quasicrystals}-\cite{matei2010simple}). We will require a version of Meyer's duality principle, we use the formulation given in \cite[Lemma 2.1]{kozma2011exponential}.

\begin{lem}[\cite{kozma2011exponential}] \label{l4.4}
	Let $U\subset [0,1]$ be a semi-closed multiband set, and $V\subset [0,1]$ a multiband set. Then:
	\begin{enumerate}
		\item If $E\left(2\pi\Lambda\left(\alpha,V\right)\right)$ is a frame in $L^{2}\left(U\right)$, then $E\left(-2\pi\Lambda\left(\alpha,U\right)\right)$ is a Riesz sequence in $L^{2}\left(V\right)$.
		\item If $E\left(2\pi\Lambda\left(\alpha,V\right)\right)$ is a Riesz sequence in $L^{2}\left(U\right)$, then $E\left(-2\pi\Lambda\left(\alpha,U\right)\right)$ is a frame in $L^{2}\left(V\right)$.
	\end{enumerate}
\end{lem}

Here, a set is called multiband if it is a union of finitely many intervals. A multiband set is called semi-closed if its indicator function is continuous either from the left or from the right. 
Combining Lemma \ref{l4.4} and Beurling's sufficient condition for frames \cite[Theorem 3.33]{olevskii2016functions} together with a rescaling argument gives the following universality result of Matei and Meyer \cite{matei2009variant}, see \cite[Section 6.6.5]{olevskii2016functions}.

\begin{thm}[\cite{matei2009variant}]\label{t4.5}
	The system $E\left(\Lambda\left(\alpha,I\right)\right)$
	is a universal Riesz sequence, i.e., $E\left(\Lambda\left(\alpha,I \right)\right)$
	is a Riesz sequence in $L^{2}\left(\SS\right)$ for every open set
	$\SS\subset\TT$ with $\frac{\left|\SS\right|}{2\pi}>\left|I\right|$.
\end{thm}

\begin{proof}[Proof of Theorem \ref{B}]
Let $n\geq2$ be an integer such that $\frac1{n} < \frac{|S|}{2\pi} \le \frac1{n-1}$. Fix an irrational $\frac{1}{n}<a<\frac{1}{n+1}$ such that 
$\frac{|\SS|}{2\pi}>a$.
Set the irrational $\alpha=\frac{1-a}{n-1}$ and define $I=\left[0,a\right)$.
Now consider the simple quasicrystal
\[
	\Lambda (\alpha,I )=\{ \lambda_{j} :j\in\ZZ \}
\qquad\text{where }\lambda_j<\lambda_{j+1} \text{ for all }j\in\ZZ. 
\]
By Theorem \ref{t4.5} $E\left(\Lambda\left(\alpha,I\right)\right)$ is a
Riesz sequence in $L^2(\SS)$ since $\frac{\left|\SS\right|}{2\pi}>\left|I\right|=a$.
Moreover, if for some
$k\in\ZZ$ 
\[
	0\leq\left\{ \alpha k\right\} <a\quad\mbox{and}\quad\,a\leq\left\{ \alpha\left(k+1\right)\right\} <1 
	\]
	then by the choice of $\alpha$ we have 
\[
	a\leq\left\{\alpha\left(k+j\right)\right\} <1 \qquad \mbox{for } j\in\{2,\ldots,n-1\}.
\]
Since $\alpha<a$ we also have
\[
	\left\{ \alpha\left(k+n\right)\right\} \in I. 
\]
This implies that 
\[
	\lambda_{j+1}-\lambda_{j}\in\left\{ 1,n\right\} \qquad\text{for all }j\in\mathbb Z.
\]

For the moreover part we let $n\ge 2$ be such that $1-\frac1{n} < \frac{|S|}{2\pi} \le  1- \frac1{n+1}$. Fix an irrational $\frac{1}{n+1}<a<\frac{1}{n}$ such that $\frac{|S|}{2\pi} > 1-a$. We set $\alpha=\frac{1-a}{n-1}$ and $I=[0,a)$ as above. Thus, by Theorem \ref{t4.5} $E\left(\Lambda\left(\alpha,I^{c}\right)\right)$ is a Riesz sequence in $L^2(\SS)$, where $I^{c}=[a,1)$. Since $\alpha>a$, if $\{\alpha k\} \in I$ for some $k\in I$, then $\{\alpha(k+1)\} \in I^c$ and consequently we have
\[
\{\alpha (k+j) \} \in I^c \qquad \text{for } j\in\{1,\ldots,n-1\}.
\]
Hence, the difference between consecutive elements in $\Lambda(\alpha,I)$ is bounded below by $n$.
Since
\[
	\Lambda\left(\alpha,I^{c}\right)=\ZZ\backslash\Lambda\left(\alpha,I\right)
\]
we deduce that $\Lambda\left(\alpha,I^{c}\right)$ satisfies \eqref{eq:7}.
\end{proof}

There are still a few open problems left with regards to the subject of syndetic Riesz sequences. Below we give a couple of them. \par

\textbf{Open problem 1.}
By Theorem \ref{t4.1}, $E\left(\Lambda\left(\alpha,I\right)\right)$
is not a Riesz sequence in $L^{2}\left(\SS\right)$ for
some non-open set $\SS \subset \TT$ with $\frac{\left|\SS\right|}{2\pi}>\left|I\right|$. Does there exist $\SS$ with 
empty interior for which $E\left(\Lambda\left(\alpha,I\right)\right)$ 
is a Riesz sequence? Otherwise prove that such set does not exist.

\textbf{Open problem 2.} Does the moreover part of
Theorem \ref{B} hold in the general case? That is, an arbitrary measurable set $\SS \subset \TT$
with $\frac{|\SS|}{2\pi}>\frac{1}{2}$ admits an exponential Riesz sequence
$E\left(\Lambda\right)$ such that $\Lambda^{c}=\ZZ\backslash\Lambda$
satisfies
\[
\inf_{\lambda,\mu\in\Lambda^{c},\lambda\neq\mu}\left|\lambda-\mu\right|\geq\frac{C}{\left|\SS^{c}\right|}.
\]
\par

\end{document}